\documentclass[oneside, 12pt]{amsart}
\usepackage{amsfonts,amsmath,amsthm,amscd,amssymb,latexsym,amsbsy,cite}

\usepackage{mathtools}
\mathtoolsset{showonlyrefs}

\usepackage{easyReview}

\usepackage{enumitem}

\usepackage{tikz}
\usetikzlibrary{positioning, calc, trees, decorations.markings, patterns, arrows, arrows.meta}

\usepackage[a4paper, top=2cm, bottom=2cm, left=3cm, right=2cm, foot=1cm]{geometry}

\usepackage[mathlines, right]{lineno}

\newtheorem{theorem}{Theorem}[section]

\newtheorem{corollary}[theorem]{Corollary}
\newtheorem{proposition}[theorem]{Proposition}

\theoremstyle{definition}
\newtheorem{definition}[theorem]{Definition}

\theoremstyle{remark}
\newtheorem{example}[theorem]{Example}
\newtheorem{remark}[theorem]{Remark}

\makeatletter
\@namedef{subjclassname@2020}{%
  \textup{2020} Mathematics Subject Classification}
\makeatother

\numberwithin{equation}{section}

\newcommand{\acr}{\newline\indent}

\author{Oleksiy Dovgoshey}
\address{\textbf{Oleksiy Dovgoshey}\acr
Department of Theory of Functions,\acr
Institute of Applied Mathematics and Mechanics of NASU, Slovyansk, Ukraine, \acr
and \acr
Department of Mathematics and Statistics, \acr
University of Turku, Turku, Finland.}
\email{oleksiy.dovgoshey@gmail.com, oleksiy.dovgoshey@utu.fi}

\author{Ruslan Shanin}
\address{\textbf{Ruslan Shanin}\acr
Department of Mathematical Analysis, \acr
Odesa I.~I.~Mechnikov National University, Odesa, Ukraine}

\email{ruslanshanin@gmail.com}

\title{On the closure of one point sets in \(T_0\)-spaces}



\newcommand{\cl}{\operatorname{cl}}

\begin{document}

\begin{abstract}
Let \(X\) be a set and \(2^X\) be a set of all subsets of \(X\). The necessary and sufficient conditions under which a mapping \(X \to 2^X\) is a closure of one-point sets in some \(T_0\)-space \((X, \tau)\) are described. It is proved that every \(T_0\)-Alexandroff space is quasi-metrizable by some equidistant quasi-metric.
\end{abstract}

\maketitle

\section{Introduction}

In this paper, we study the closure of one-point sets in \(T_0\)-spaces and, in particular, in \(T_0\)-Alexandroff spaces.

The Alexandroff spaces were first considered in 1937 by P. S. Alexandroff \cite{Ale1937DR} under the name ``Diskrete Ra\"{u}me''. The first systematic study of Alexandroff spaces was carried by F. G. Arenas \cite{Are1999AS} for the case of arbitrary, not necessarily \(T_0\)-spaces. In the present paper we mainly consider the \(T_0\)-Alexandroff spaces which were studied, for example, in \cite{ME2005OAS}.

Some results connected with \(T_0\)-spaces can be found in~\cite{Lar1969MSAMS, Gou2013NHTADTSTIP, SAZS2020TOS, YX2025OCIS, JL2025NOSCALYC, OM2022OBBSOQPS, LZ2023LYCOQMS}. In particular, book~\cite{Gou2013NHTADTSTIP} contains useful information on quasi-metrics and \(T_0\)-spaces generated them. The \(T_0\)-quasi-metric spaces were also considered by authors of \cite{YK2019SCIQMS, KYJ2023SAADSOQMS, JY2020SCAACQME, JY2025FAOLACTQMS}.

It was proved by Zbigniew Karno~\cite{Karno2003} that a topological space satisfies the \(T_0\)-separation axiom if and only if any two distinct points of this space have distinct closures. Using this fact we describe in Theorem~\ref{t3.1} the necessary and sufficient conditions under which a mapping \(X \to 2^X\) can be extended to the closure operator in some \(T_0\)-space \((X, \tau)\). The same Theorem~\ref{t3.1} shows that the extension of \(X \to 2^X\) to the closure operator in \(T_0\)-topology on \(X\) is possible if and only if such extension exists in \(T_0\)-Alexandroff topology on \(X\).

In Theorem~\ref{t3.4} it is shown that every \(T_0\)-Alexandroff space is quasi-metrizable by some equidistant quasi-metric.

\section{Preliminaries}

Let us recall a method of generating topologies by closure operator. In what follows the symbol \(\cl_{\tau}(A)\) denotes the closure of a set \(A \subseteq X\) in a topological space \((X, \tau)\) and we write \(2^X\) for the set of all subsets of a set \(X\).

\begin{theorem}\label{t1.1}
Let \(X\) be a set and let a mapping \(\cl \colon 2^X \to 2^X\) satisfy the conditions:
\begin{enumerate}[label=\ensuremath{(c_\arabic*)}]
\item\label{t1.1:s1} \(\cl(\varnothing) = \varnothing\),
\item\label{t1.1:s2} \(A \subseteq \cl(A)\),
\item\label{t1.1:s3} \(\cl(\cl(A)) = \cl(A)\),
\item\label{t1.1:s4} \(\cl(A \cup B) = \cl(A) \cup \cl(B)\)
\end{enumerate}
for all \(A\), \(B \subseteq X\). Then there exists a topology \(\tau\) on \(X\) such that
\begin{equation}\label{t1.1:e1}
\cl(A) = \cl_{\tau}(A)
\end{equation}
for each \(A \subseteq X\).
\end{theorem}

For the proof see, for example, Proposition~1.2.7 in \cite{Engelking1989}.

\begin{remark}\label{r1.2}
The topology \(\tau\) satisfying~\eqref{t1.1:e1} for each \(A \subseteq X\) is unique, because any two topologies coincide if their closed sets are the same. Since the closure operator \(\cl = \cl_{\tau}\) satisfies conditions \ref{t1.1:s1}--\ref{t1.1:s4} for every topological space \((X, \tau)\), each topology on \(X\) can be generated using the method described in Theorem~\ref{t1.1}.
\end{remark}

\begin{definition}\label{d1.3}
A topological space \((X, \tau)\) is said to be \(T_0\)-space if for any two distinct points \(x\), \(y \in X\) there exists an open set \(O \in \tau\) such that either
\begin{equation*}
x \in O \text{ and } y \in X \setminus O
\end{equation*}
or
\begin{equation*}
x \in X \setminus O \text{ and } y \in O .
\end{equation*}
\end{definition}

\begin{remark}\label{r1.3.1}
Alexandroff and Hopf call \(T_0\)-topological spaces Kolmogorov spaces \cite[p.~58]{AH1935TBIGDMTTDK}.
\end{remark}

The next proposition gives us an equivalent reformulation of Definition~\ref{d1.3}.

\begin{proposition}\label{p2.1}
The following statements are equivalent for every topological space \((X, \tau)\):
\begin{enumerate}[label=\ensuremath{(\roman*)}]
\item\label{p2.1:s1} \((X, \tau)\) is a \(T_0\)-space.
\item\label{p2.1:s2} For any two distinct points \(x\), \(y \in X\) we have
\begin{equation*}
\cl_{\tau}(\{x\}) \neq \cl_{\tau}(\{y\}).
\end{equation*}
\end{enumerate}
\end{proposition}

For the proof of Proposition~\ref{p2.1} see paper \cite{Karno2003}.

\begin{definition}\label{d1.3.1}
A topological space \((X, \tau)\) is called the Alexandroff space if the intersection of arbitrary family of open subsets of \((X, \tau)\) is open.
\end{definition}

The following proposition is a simple consequence of Definition~\ref{d1.3.1}.

\begin{proposition}\label{p1.3.2}
A topological space \((X, \tau)\) is an Alexandroff space if and only if the union of arbitrary family of closed subsets of \((X, \tau)\) is closed.
\end{proposition}

The next result seems to be know, but the authors cannot give a precise reference here.

\begin{proposition}\label{p2.3}
Let \((X, \tau)\) be a topological space. Then the equality
\begin{equation}\label{p2.3:e1}
\cl_{\tau}(A) = \bigcup_{x \in A} \cl_{\tau}(\{x\})
\end{equation}
holds for every nonempty \(A \subseteq X\) if and only if \((X, \tau)\) is an Alexandroff space.
\end{proposition}

\begin{proof}
Suppose that \((X, \tau)\) is an Alexandroff space. Let \(A\) be a nonempty subset of \(X\). We must prove that~\eqref{p2.3:e1} holds. By Proposition~\ref{p1.3.2}, the union of arbitrary family of closed subsets of \((X, \tau)\) is closed. Hence, the set \(\cup_{x \in A} \cl_{\tau}(\{x\})\) is closed in \((X, \tau)\).

The closure \(\cl_{\tau}(A)\) is the smallest closed set containing the set \(A\). Thus,
\begin{equation}\label{p2.3:e2}
\cl_{\tau}(A) \subseteq \bigcup_{x \in A} \cl_{\tau}(\{x\})
\end{equation}
holds, because \(A \subseteq \cup_{x \in A} \cl_{\tau}(\{x\})\). Theorem~\ref{t1.1} gives us the inclusion
\[
\cl_{\tau}(\{x\}) \subseteq \cl_{\tau}(A)
\]
for every \(x \in A\). Consequently,
\begin{equation}\label{p2.3:e3}
\bigcup_{x \in A} \cl_{\tau}(\{x\}) \subseteq \cl_{\tau}(A)
\end{equation}
holds. Equality~\eqref{p2.3:e1} follows from \eqref{p2.3:e2} and \eqref{p2.3:e3}.

Let~\eqref{p2.3:e1} hold for every nonempty \(A \subseteq X\) and let \(\{A_j\}_{j \in J}\) be a family of nonempty closed subsets of \((X, \tau)\). Write
\[
B := \bigcup_{j \in J} A_j.
\]
Then \(A_j = \cl_{\tau}(A_j)\) for all \(j \in J\) and we have
\[
B = \bigcup_{j \in J} A_j = \bigcup_{j \in J} \cl_{\tau}(A_j) = \bigcup_{j \in J} \left(\bigcup_{x \in A_j} \cl_{\tau}(\{x\})\right) = \bigcup_{x \in B} \cl_{\tau}(\{x\}) = \cl_{\tau}(B).
\]
Hence, \(B\) is closed and, consequently, \((X, \tau)\) is an Alexandroff space by Proposition~\ref{p1.3.2}.
\end{proof}

Proposition~\ref{p2.3} and Remark~\ref{r1.2} imply the following corollary.

\begin{corollary}\label{c2.4}
Let \((X, \tau_1)\) and \((X, \tau_2)\) be Alexandroff spaces. Then the topologies \(\tau_1\) and \(\tau_2\) coincide if and only if the equality
\[
\cl_{\tau_1}(\{x\}) = \cl_{\tau_2}(\{x\})
\]
holds for each \(x \in X\).
\end{corollary}

Let us recall now the concept of quasi-metric spaces. The following definition is a special case of Definition~6.1.1 from~\cite{Gou2013NHTADTSTIP}.

\begin{definition}\label{d1.4}
Let \(X\) be a set. A mapping \(d \colon X \times X \to \mathbb{R}^+\) is said to be a quasi-metric on \(X\) if the following conditions hold:
\begin{enumerate}[label=\ensuremath{(c_\arabic*)}]
\item\label{d1.4:c1} \(d(x, x) = 0\) for every \(x \in X\);
\item\label{d1.4:c2} (Triangle inequality)
\[
d(x, y) \leqslant d(x, z) + d(z, y)
\]
for all \(x\), \(y\), \(z \in X\);
\item\label{d1.4:c3} \(d(x, y) = d(y, x) = 0\) implies \(x = y\) for all \(x\), \(y \in X\).
\end{enumerate}
A set \(X\) with a quasi-metric \(d \colon X \times X \to \mathbb{R}^+\) is said to be a quasi-metric space and the quantity \(d(x, y)\) is called the distance from \(x\) to \(y\).
\end{definition}

In general, the quasi-metrics are not symmetric that allows us to distinguish between the distances \(d(x,y)\) and \(d(y, x)\).

Similarly to the usual metric spaces we can define the open balls in all quasi-metric spaces.

\begin{definition}\label{d1.6}
Let \(d\) be a quasi-metric on a set \(X\). The open ball with center \(x \in X\) and radius \(r \in (0, \infty)\) is the set of all points \(y \in X\) such that \(d(x, y) < r\).
\end{definition}

The open ball topology \(O^d\) on \(X\) is, by definition, the topology generated by the set of all open balls of the quasi-metric space \((X, d)\). The set of all open balls of a quasi-metric space \((X, d)\) forms a base of the topology \(O^d\).

\begin{definition}\label{d1.6.1}
A topological space \((X, \tau)\) is quasi-metrizable if there is a quasi-metric \(d \colon X \times X \to \mathbb{R}^+\) such that \(\tau = O^d\).
\end{definition}

The following proposition directly follows from Lemma~6.1.9 of~\cite{Gou2013NHTADTSTIP}.

\begin{proposition}\label{p1.7}
Let \((X, \tau)\) be a quasi-metrizable topological space. Then \((X, \tau)\) is a \(T_0\)-space.
\end{proposition}

The closure of subsets of quasi-metric spaces can be described as follows.

\begin{proposition}\label{p1.8}
Let \(A\) be a nonempty subset of a quasi-metric space \((X, d)\), let \(x \in X\), and let
\begin{equation}\label{p1.8:e1}
d(x, A) := \inf_{a \in A} d(x, a).
\end{equation}
Then the closure of the set \(A\) in the topological space \((X, O^d)\) coincides with the set of all points \(x\) of \(X\) which satisfy the equality
\begin{equation}\label{p1.8:e2}
d(x, A) = 0.
\end{equation}
\end{proposition}

For the proof see Lemma~6.1.11 of~\cite{Gou2013NHTADTSTIP}.

\begin{corollary}\label{c1.9}
Let \(p\) be a point of a quasi-metric space \((X, d)\). Then the equality
\begin{equation}\label{c1.9:e1}
\cl_{O^d} (\{p\}) = \{x \in X \colon d(x, p) = 0\}
\end{equation}
holds.
\end{corollary}

\begin{proof}
It follows from~\eqref{p1.8:e1} and \eqref{p1.8:e2} with \(A = \{p\}\).
\end{proof}

Following~\cite[p.~46]{DD2016EOD}, we say that a metric space \((X, d)\) is equidistant if there exists \(t > 0\) such that \(d(x, y) = t\) for all distinct \(x\), \(y \in X\). By analogy with equidistant metrics we can introduce the concept of equidistant quasi-metrics.

\begin{definition}\label{d1.10}
Let \((X, d)\) be a quasi-metric space. The quasi-metric \(d\) is equidistant if there exists \(t > 0\) such that \(d(x, y) = t\) whenever \(d(x, y) \neq 0\).
\end{definition}

It is easy to see that a quasi-metric \(d \colon X \times X \to \mathbb{R}^+\) is equidistant if and only if the mapping
\[
X \times X \ni (x, y) \mapsto \max\{d(x, y), d(y, x)\}
\]
is an equidistant metric on \(X\).

\begin{remark}\label{r1.11}
Some results closely connected with equidistant metrics and their pseudometric generalization can be found in~\cite{Dov2025HDBUB, BD2023PSFMTMITGO, BD2023WAPACS, BD2024OMOMPF, DL2020CCOP}.
\end{remark}

\section{The closure of points in \(T_0\)-spaces and Alexandroff \(T_0\)-spaces}

In what follows we denote by \(\boldsymbol{\tau}_0(X)\) the set of all \(T_0\)-topologies on a given set \(X\).

Let \(X\) be an arbitrary set. The following theorem gives us necessary and sufficient conditions under which the mapping \(X \rightarrow 2^X\) can be extended to a closure operator \(2^X \rightarrow 2^X\) in a \(T_0\)-space \((X, \tau)\).

\begin{theorem}\label{t3.1}
Let \(X\) be a nonempty set and let \(\cl_1\colon X \rightarrow 2^X\) be a mapping. Then the following statements are equivalent:
\begin{enumerate}
\item \label{t3.1:s1} There exists \(\tau \in \boldsymbol{\tau}_0(X)\) such that the equality
\begin{equation}\label{t3.1:e1}
\cl_{\tau}(\{x\}) = \cl_1(x)
\end{equation}
holds for each \(x \in X\).

\item \label{t3.1:s2} The mapping \(\cl_1\colon X \rightarrow 2^X\) is injective, and the membership relation
\begin{equation}\label{t3.1:e2}
x \in \cl_1(x)
\end{equation}
is valid for each \(x \in X\), and the inclusion
\begin{equation}\label{t3.1:e3}
\cl_1(y) \subseteq \cl_1(x)
\end{equation}
holds for every \(x \in X\) and every \(y \in \cl_1(x)\).

\item \label{t3.1:s3} There exists the unique Alexandroff topology \(\tau \in \boldsymbol{\tau}_0(X)\) such that~\eqref{t3.1:e1} holds for each \(x \in X\).
\end{enumerate}
\end{theorem}

\begin{proof}
\(\ref{t3.1:s1} \Rightarrow \ref{t3.1:s2}\). Let \ref{t3.1:s1} hold. Equality~\eqref{t3.1:e1} and Proposition~\ref{p2.1} imply that the mapping \(\cl_1\colon X \rightarrow 2^X\) is injective. Moreover, condition~\ref{t1.1:s2} of Theorem~\ref{t1.1} implies that~\eqref{t3.1:e2} holds for each \(x \in X\).

Let us consider now arbitrary points \(x \in X\) and
\begin{equation}\label{t3.1:e4}
y \in \cl_1(x).
\end{equation}
Relation~\eqref{t3.1:e4} and equality~\eqref{t3.1:e1} give us the inclusion
\begin{equation}\label{t3.1:e5}
\{y\} \subseteq \cl_1(x) = \cl_{\tau}(\{x\}).
\end{equation}
Condition~\ref{t1.1:s4} of Theorem~\ref{t1.1} with \(\cl = \cl_{\tau}\) implies
\begin{equation}\label{t3.1:e6}
\cl_{\tau}(A) \subseteq \cl_{\tau}(B)
\end{equation}
whenever \(A \subseteq B\). Using~\eqref{t3.1:e6} with \(A = \{y\}\) and \(B = \cl_{\tau}(\{x\})\) we obtain
\[
\cl_{\tau}(\{y\}) \subseteq \cl_{\tau}(\cl_{\tau}(\{x\})).
\]
The last inclusion and condition~\ref{t1.1:s3} of Theorem~\ref{t1.1} imply
\begin{equation}\label{t3.1:e7}
\cl_{\tau}(\{y\}) \subseteq \cl_{\tau}(\{x\}).
\end{equation}
Now~\eqref{t3.1:e3} follows from statement~\ref{t3.1:s1} and inclusion~\eqref{t3.1:e7}. Thus, the implication \(\ref{t3.1:s1} \Rightarrow \ref{t3.1:s2}\) is valid.

\(\ref{t3.1:s2} \Rightarrow \ref{t3.1:s3}\). Let \ref{t3.1:s2} hold. For every \(A \subseteq X\) we define a mapping \(\cl \colon 2^X \to 2^X\) as
\begin{equation}\label{t3.1:e8}
\cl(A) := \begin{cases}
A, & \text{if } A = \varnothing,\\
\bigcup_{x \in A} \cl_1(x), & \text{otherwise}.
\end{cases}
\end{equation}
We claim that there is a topology \(\tau\) on \(X\) such that
\begin{equation}\label{t3.1:e8.1}
\cl_{\tau}(A) = \cl(A)
\end{equation}
for every \(A \subseteq X\). To prove this claim we will use Theorem~\ref{t1.1}.

First of all we note that conditions~\ref{t1.1:s1} and \ref{t1.1:s4} of Theorem~\ref{t1.1} directly follow from~\eqref{t3.1:e8}. In addition, using~\eqref{t3.1:e2} and \eqref{t3.1:e8} we obtain the inclusion
\[
A \subseteq \cl(A)
\]
for every \(A \subseteq X\). Thus, condition~\ref{t1.1:s2} of Theorem~\ref{t1.1} is also satisfied.

It remains to check that condition~\ref{t1.1:s3} of Theorem~\ref{t1.1} is satisfied. To do this, we need to prove the equality
\begin{equation}\label{t3.1:e9}
\cl(\cl(A)) = \cl(A)
\end{equation}
for arbitrary \(A \subseteq X\).

If \(A\) is empty subset of \(X\), then \eqref{t3.1:e9} directly follows from the equality \(\cl(\varnothing) = \varnothing\) which was proved above.

Let us consider the case when \(A \neq \varnothing\). From \eqref{t3.1:e3} and \eqref{t3.1:e8} it follows that for each \(y \in \cl(A)\) there exists \(x_y \in A\) such that \(y \in \cl_1(x_y)\) and \(\cl_1(y) \subseteq \cl_1(x_y)\). Hence, we obtain the inclusion \(\cl_1(y) \subseteq \cl(A)\) for each \(y \in \cl(A)\). Since \(A \subseteq \cl(A)\), we have
\begin{equation*}
\begin{aligned}
\cl(A)  & = \bigcup_{x \in A} \cl_1(x) \subseteq \bigcup_{x \in \cl(A)} \cl_1(x) \\
& = \cl(\cl(A)) = \bigcup_{y \in \cl(A)} \cl_1(y) \subseteq \bigcup_{y \in \cl(A)} \cl(A) = \cl(A).
\end{aligned}
\end{equation*}
Equality~\eqref{t3.1:e9} follows for arbitrary \(A \subseteq X\).

Consequently, by Theorem~\ref{t1.1}, there exists a topology \(\tau\) on \(X\) such that~\eqref{t3.1:e8.1} holds for each \(A \subseteq X\).

Let us prove that \(\tau \in \boldsymbol{\tau}_0(X)\). Using equality~\eqref{t3.1:e8} for \(A = \{x\}\) it is easy to see that~\eqref{t3.1:e1} holds for every \(x \in X\). Hence, the membership relation
\[
\tau \in \boldsymbol{\tau}_0(X)
\]
follows from~\eqref{t3.1:e1} by Proposition~\ref{p2.1} because \(\cl_1\colon X \rightarrow 2^X\) is injective by statement~\ref{t3.1:s2}.

Equalities~\eqref{t3.1:e8}, \eqref{t3.1:e8.1} and Proposition~\ref{p2.3} imply that \((X, \tau)\) is an Alexandroff space. The uniqueness of Alexandroff topology in the set of all \(T_0\)-topologies which satisfy~\eqref{t3.1:e1} for each \(x \in X\) follows from~Corollary~\ref{c2.4}.

\(\ref{t3.1:s3} \Rightarrow \ref{t3.1:s1}\). This implication is evidently valid.

The proof is completed.
\end{proof}

Let \(\tau_1\) and \(\tau_2\) be topologies on a set \(X\). If the inclusion \(\tau_1 \subseteq \tau_2\) holds, then the topology \(\tau_2\) is said to be larger (or finer) than \(\tau_1\).

\begin{proposition}\label{p3.2}
Let \(X\) be a nonempty set, \((X, \tau^A)\) be an Alexandroff topological space and let \(\tau \in \boldsymbol{\tau}_{0}(X)\). If the equality
\begin{equation}\label{p3.2:e1}
\cl_{\tau}(\{x\}) = \cl_{\tau^A}(\{x\})
\end{equation}
holds for every \(x \in X\), then \(\tau^A\) is finer than \(\tau\),
\begin{equation}\label{p3.2:e2}
\tau \subseteq \tau^A.
\end{equation}
\end{proposition}

\begin{proof}
Let \eqref{p3.2:e1} hold for every \(x \in X\). We must prove that inclusion~\eqref{p3.2:e2} holds.

To prove inclusion~~\eqref{p3.2:e2} it suffices to show that every closed in \((X, \tau)\) subset of \(X\) is also closed in the space \((X, \tau^{A})\),
\begin{equation}\label{p3.2:e3}
\cl_{\tau^A}(S) = S.
\end{equation}

Let \(S\) be an arbitrary closed in \((X, \tau)\) subset of \(X\),
\begin{equation}\label{p3.2:e4}
\cl_{\tau}(S) = S.
\end{equation}
Then condition~\ref{t1.1:s4} of Theorem~\ref{t1.1} implies the inclusion
\begin{equation*}
\cl_{\tau}(\{x\}) \subseteq \cl_{\tau}(S)
\end{equation*}
for every \(x \in S\). Now using equality~\eqref{p3.2:e1} and equality~\eqref{p2.3:e1} with \(A = S\) and \(\tau = \tau^A\) we obtain
\begin{equation}\label{p3.2:e6}
\cl_{\tau^A}(S) = \bigcup_{x \in S} \cl_{\tau^A}(\{x\}) = \bigcup_{x \in S} \cl_{\tau}(\{x\}).
\end{equation}
Inclusion \(\cl_{\tau}(\{x\}) \subseteq \cl_{\tau}(S)\), \(x \in S\), and equalities~\eqref{p3.2:e4} and \eqref{p3.2:e6} give us
\begin{equation}\label{p3.2:e7}
\cl_{\tau^A}(S) \subseteq S.
\end{equation}
The converse inclusion follows from condition~\ref{t1.1:s2} of Theorem~\ref{t1.1} with \(\cl = \cl_{\tau^A}\). Equality \eqref{p3.2:e3} follows.
\end{proof}

The next theorem is the second main result of the paper.

\begin{theorem}\label{t3.4}
Let \((X, \tau)\) be a topological space. Then the following statements are equivalent:
\begin{enumerate}
\item \label{t3.4:s1} The space \((X, \tau)\) is quasi-metrizable by equidistant quasi-metric \(d \colon X \times X \to \mathbb{R}^{+}\).
\item \label{t3.4:s2} The space \((X, \tau)\) is a \(T_0\)-Alexandroff space.
\end{enumerate}
\end{theorem}

\begin{proof}
\(\ref{t3.4:s1} \Rightarrow \ref{t3.4:s2}\). Let \ref{t3.4:s1} hold. Then there is an equidistant quasi-metric \(d \colon X \times X \to \mathbb{R}^{+}\) such that \(\tau = O^d\). The topological space \((X, O^d)\) is a \(T_0\)-space by Proposition~\ref{p1.7}. Hence, \((X, \tau)\) is also a \(T_0\)-space. Thus, to prove statement~\ref{t3.4:s2} it suffices to show that \((X, \tau)\) is an Alexandroff space.

Proposition~\ref{p2.3} implies that \((X, \tau)\) is an Alexandroff space if the equality
\begin{equation*}
\cl_{\tau}(A) = \bigcup_{a \in A} \cl_{\tau}(\{a\})
\end{equation*}
holds for every nonempty \(A \subseteq X\).

Let us consider an arbitrary nonempty \(A \subseteq X\). The equality \(\tau = O^d\) and Proposition~\ref{p1.8} give us the equality
\[
\cl_{\tau}(A) = \left\{x \in X \colon \inf_{a \in A} d(x, A) = 0\right\}.
\]
Thus, we must prove the equality
\begin{equation}\label{t3.4:e2}
\bigcup_{a \in A} \cl_{\tau}(\{a\}) = \left\{x \in X \colon \inf_{a \in A} d(x, A) = 0\right\}.
\end{equation}
To prove equality~\eqref{t3.4:e2} we first note that Corollary~\ref{c1.9} and the equality \(\tau = O^d\) imply
\begin{equation}\label{t3.4:e3}
\cl_{\tau}(\{a\}) = \left\{x \in X \colon d(x, A) = 0\right\}
\end{equation}
for every \(a \in A\). Consequently, the inclusion
\[
\bigcup_{a \in A} \cl_{\tau}(\{a\}) \subseteq \left\{x \in X \colon \inf_{a \in A} d(x, A) = 0\right\}
\]
holds. Hence, \eqref{t3.4:e2} is valid if we have
\begin{equation}\label{t3.4:e4}
\bigcup_{a \in A} \cl_{\tau}(\{a\}) \supseteq \left\{x \in X \colon \inf_{a \in A} d(x, A) = 0\right\}.
\end{equation}
Suppose the contrary that there is \(x_0 \in X\) such that
\begin{equation*}
x_0 \notin \cl_{\tau}(\{a\})
\end{equation*}
for every \(a \in A\). Then, using Definition~\ref{d1.10} and equality~\eqref{t3.4:e3}, we can find \(t \in (0, \infty)\) such that
\[
d(x_0, a) = t
\]
for all \(a \in A\). It implies
\[
\inf_{a \in A} d(x_0, A) = t > 0.
\]
Thus, if we have
\[
x_0 \notin \bigcup_{a \in A} \cl_{\tau}(\{a\}),
\]
then the relation
\[
x_0 \notin \left\{x \in X \colon \inf_{a \in A} d(x, A) = 0\right\}
\]
holds. Inclusion~\eqref{t3.4:e4} follows. Hence, \((X, \tau)\) is an \(T_0\)-Alexandroff space.

\(\ref{t3.4:s2} \Rightarrow \ref{t3.4:s1}\). Let \((X, \tau)\) be a \(T_0\)-Alexandroff space. Let us define a mapping \(d \colon X \times X \to \mathbb{R}^+\) as
\begin{equation}\label{t3.4:e6}
d(x, y) := \begin{cases}
0 & \text{if } x \in \cl_{\tau}(\{y\}),\\
1 & \text{otherwise}.
\end{cases}
\end{equation}

We claim that \(d \colon X \times X \to \mathbb{R}^+\) is an equidistant quasi-metric and that \(O^d = \tau\) holds.

The mapping \(d\) is a quasi-metric on \(X\) iff conditions~\ref{d1.4:c1}--\ref{d1.4:c3} of Definition~\ref{d1.4} are satisfied.

Since \(x\) belongs to \(\cl_{\tau}(\{x\})\) for each \(x \in X\), condition~\ref{d1.4:c1} of Definition~\ref{d1.4} follows from~\eqref{t3.4:e6}.

Let us prove the validity of~\ref{d1.4:c2} (the triangle inequality). Suppose the contrary that there are \(x\), \(y\), \(z \in X\) such that
\begin{equation}\label{t3.4:e7}
d(x, y) > d(x, z) + d(z, y).
\end{equation}
Then using~\eqref{t3.4:e6} and \eqref{t3.4:e7} we obtain
\[
d(x, y) = 1 \text{ and } d(x, z) = d(z, y) = 0,
\]
that implies
\begin{equation}\label{t3.4:e8}
x \notin \cl_{\tau}(\{y\})
\end{equation}
and
\begin{equation}\label{t3.4:e9}
x \in \cl_{\tau}(\{z\}), \ z \in \cl_{\tau}(\{y\}).
\end{equation}
Membership relations~\eqref{t3.4:e9} and Theorem~\ref{t1.1} give us
\[
\cl_{\tau}(\{x\}) \subseteq \cl_{\tau}(\cl_{\tau}(\{z\})) = \cl_{\tau}(\{z\})
\]
and
\[
\cl_{\tau}(\{z\}) \subseteq \cl_{\tau}(\cl_{\tau}(\{y\})) = \cl_{\tau}(\{y\}),
\]
that implies
\begin{equation}\label{t3.4:e10}
\cl_{\tau}(\{x\}) \subseteq \cl_{\tau}(\{y\}).
\end{equation}
Inclusion~\eqref{t3.4:e10} and the membership relation \(x \in \cl_{\tau}(\{x\})\) imply
\[
x \in \cl_{\tau}(\{y\}),
\]
contrary to~\eqref{t3.4:e8}. Thus, the triangle inequality holds for all \(x\), \(y\), \(z \in X\).

Let us prove condition~\ref{d1.4:c3} of Definition~\ref{d1.4}. Let \(x\) and \(y\) be points of \(X\) such that
\[
d(x, y) = d(y, x) = 0.
\]
Then
\[
x \in \cl_{\tau}(\{y\}) \text{ and } y \in \cl_{\tau}(\{x\})
\]
hold by~\eqref{t3.4:e6}. Since \((X, \tau)\) is a \(T_0\)-space by statement~\ref{t3.4:s2}, we have the equality \(x = y\) by Proposition~\ref{p2.1}. Thus, condition~\ref{d1.4:c3} of Definition~\ref{d1.4} is valid.

Hence, the mapping \(d\) defined by formula~\eqref{t3.4:e6} is a quasi-metric on \(X\).

To complete the proof we must show that the equality
\begin{equation}\label{t3.4:e11}
O^d = \tau
\end{equation}
holds.

It is trivially valid that the topological space \((X, O^d)\) is quasi-metrizable by equidistant quasi-metric \(d\). Moreover, it was proved in the first part of the proof that the implication \(\ref{t3.4:s1} \Rightarrow \ref{t3.4:s2}\) is true. Consequently, \((X, O^d)\) is a \(T_0\)-Alexandroff space. Equality~\eqref{t3.4:e6} and Corollary~\ref{c1.9} give us the equality
\begin{equation}\label{t3.4:e12}
\cl_{\tau}(\{x\}) = \cl_{O^d}(\{x\}).
\end{equation}
for each \(x \in X\). Now~\eqref{t3.4:e11} follows from~\eqref{t3.4:e12} and Corollary~\ref{c2.4}.

The proof is completed.
\end{proof}

Let \(X\) be a given set and let \(\tau_1\) and \(\tau_2\) be two \(T_0\)-topologies on \(X\). We will write \(\tau_1 \approx \tau_2\) if and only if the equality
\[
\cl_{\tau_1}(\{x\}) = \cl_{\tau_2}(\{x\})
\]
holds for each \(x \in X\). Then \(\approx\) is an equivalence relation on the set \(\boldsymbol{\tau}_0(X)\) of all \(T_0\)-topologies on \(X\). The quotient set of \(\boldsymbol{\tau}_0(X)\) with respect to the relation \(\approx\) is the set \(\{[\tau^*]_{\approx} \colon \tau^* \in \boldsymbol{\tau}_0(X)\}\) of all equivalence classes \([\tau^*]_{\approx}\), where
\[
[\tau^*]_{\approx} := \{\tau \in \boldsymbol{\tau}_0(X) \colon \tau \approx \tau^*\}.
\]

Theorem~\ref{t3.1}, Proposition~\ref{p3.2} and Theorem~\ref{t3.4} imply the following corollary.

\begin{corollary}\label{c3.4}
Let \(X\) be a set. Then the following statements hold for every \(\tau^* \in \boldsymbol{\tau}_0(X)\):
\begin{enumerate}
\item There exists the unique Alexandroff topology \(\tau_1\) such that \(\tau_1 \in [\tau^*]_{\approx}\).
\item There exists the unique topology \(\tau_2 \in [\tau^*]_{\approx}\) such that \(\tau \subseteq \tau_2\) for every \(\tau \in [\tau^*]_{\approx}\).
\item There exists the unique topology \(\tau_3 \in [\tau^*]_{\approx}\) such that \((X, \tau_3)\) is quasi-metrizable by equidistant quasi-metric.
\end{enumerate}
Furthermore, the topologies \(\tau_1\), \(\tau_2\) and \(\tau_3\) are the same.
\end{corollary}

Let us consider now two examples of the quotient sets \([\tau^*]_{\approx}\).

\begin{example}\label{ex3.5}
Let \(X\) be a finite set. Then all quotient sets \([\tau^*]_{\approx}\) of the set \(\boldsymbol{\tau}_0(X)\) are singletons, i.e., the equality
\begin{equation}\label{ex3.5:e1}
[\tau^*]_{\approx} = \{\tau^*\}
\end{equation}
holds for every \(\tau^* \in \boldsymbol{\tau}_0(X)\). To prove~\eqref{ex3.5:e1} we only note that every topology on finite set \(X\) is an Alexandroff topology by Definition~\ref{d1.3.1}, because arbitrary set of subsets of \(X\) is also finite. Hence, \eqref{ex3.5:e1} follows from Corollary~\ref{c3.4}.
\end{example}

\begin{example}\label{ex3.6}
Let \(X\) be a set and let \(\boldsymbol{\tau}_1(X)\) and \(\boldsymbol{\tau}_2(X)\) be the set of all Frechet topologies on \(X\) and, respectively, the set of all Hausdorff topologies on \(X\). If \(\tau^*\) is an Hausdorff topology on \(X\), then the equality
\begin{equation}\label{ex3.6:e1}
[\tau^*]_{\approx} = \boldsymbol{\tau}_2(X)
\end{equation}
holds. Moreover, the unique Alexandroff topology \(\tau_1 \in [\tau^*]_{\approx}\), which exists by Corollary~\ref{c3.4}, is metrizable by equidistant metric on \(X\). To prove equality~\eqref{ex3.6:e1} and metrizability of \(\tau_1\) by equidistant metric we only note that the inclusion
\[
\boldsymbol{\tau}_2(X) \subseteq \boldsymbol{\tau}_1(X)
\]
holds and that every singleton \(\{x\}\), \(x \in X\), is a closed subset of \((X, \tau)\) for each \(\tau \in \boldsymbol{\tau}_1(X)\).
\end{example}

\section*{Funding}

Oleksiy Dovgoshey was supported by grant 359772 of the Academy of Finland.

\bibliographystyle{plain}
\bibliography{biblio}

\end{document}